\documentclass[12pt]{amsart}
\usepackage{}
\usepackage{amsmath}
\usepackage{amsfonts}
\usepackage{amssymb}
\usepackage[all]{xy}           %xypic macro for latex2.09
\usepackage{bm}
\usepackage{bbding}
\usepackage{txfonts}
\usepackage{amscd}
\usepackage{xspace}
\usepackage[shortlabels]{enumitem}
\usepackage{ifpdf}

\ifpdf
  \usepackage[colorlinks,final,backref=page,hyperindex]{hyperref}
\else
  \usepackage[colorlinks,final,backref=page,hyperindex,hypertex]{hyperref}
\fi
\usepackage{tikz}
\usepackage[active]{srcltx}

\makeatletter

\newtheorem{thm}{Theorem}[section]
\newtheorem{lem}[thm]{Lemma}
\newtheorem{cor}[thm]{Corollary}
\newtheorem{pro}[thm]{Proposition}
\newtheorem{ex}[thm]{Example}
\theoremstyle{definition}
\newtheorem{rmk}[thm]{Remark}
\newtheorem{defi}[thm]{Definition}

\newcommand{\nc}{\newcommand}
\newcommand{\delete}[1]{}

\nc{\mlabel}[1]{\label{#1}}  % Use this to suppress names
\nc{\mcite}[1]{\cite{#1}}  % Use this to suppress names
\nc{\mref}[1]{\ref{#1}}  % Use this to suppress names
\nc{\mbibitem}[1]{\bibitem{#1}} % Use this to show number
\delete{% Use the next two lines to show names
\nc{\mlabel}[1]{\label{#1}{\hfill \hspace{1cm}{\bf{{\ }\hfill(#1)}}}}
\nc{\mcite}[1]{\cite{#1}{{\bf{{\ }(#1)}}}}  % Use this lines to show names
\nc{\mref}[1]{\ref{#1}{{\bf{{\ }(#1)}}}}  % Use this lines to show names
\nc{\mbibitem}[1]{\bibitem[\bf #1]{#1}} % Use this to show name
}

\newcommand {\emptycomment}[1]{}

\delete{
\setlength{\baselineskip}{1.8\baselineskip}

\newcommand{\emptycomment}[1]{}

}

\nc{\calo}{\mathcal{O}}
\nc{\oop}{$\mathcal{O}$-operator\xspace}
\nc{\oops}{$\mathcal{O}$-operators\xspace}
\nc{\mrho}{{\bm{\varrho}}}
\nc{\bfk}{\mathbf{K}}
\nc{\invlim}{\displaystyle{\lim_{\longleftarrow}}\,}
\nc{\ot}{\otimes}
\nc{\CV}{\mathbf{C}}

\newcommand{\lon }{\,\rightarrow\,}
\newcommand{\be }{\begin{equation}}
\newcommand{\ee }{\end{equation}}

\newcommand{\g}{\mathfrak g}

\newcommand{\huaB}{\mathcal{B}}%{{\mathcal{E}}}%{\mathcal{B}}

\newcommand{\huaH}{\mathcal{H}}

\newcommand{\huaO}{{\mathcal{O}}}

\newcommand{\huaZ}{\mathcal{Z}}

\newcommand{\frkT}{\mathfrak T}

\newcommand{\half}{\frac{1}{2}}

\newcommand{\Id}{{\rm{Id}}}

\newcommand{\br}[1]{   [ \cdot,    \cdot  ]   }

\newcommand{\Hom}{\mathrm{Hom}}

\newcommand{\Nij}{\mathrm{Nij}}
\newcommand{\Ob}{\mathsf{Ob}}

\newcommand{\gl}{\mathfrak {gl}}

\newcommand{\K}{\mathbf{K}}

\begin{document}

\title[Deformations of relative Rota-Baxter operators on Leibniz algebras]{Deformations of relative Rota-Baxter operators on Leibniz algebras}

\author{Rong Tang}
\address{Department of Mathematics, Jilin University, Changchun 130012, Jilin, China}
\email{tangrong@jlu.edu.cn}

\author{Yunhe Sheng}
\address{Department of Mathematics, Jilin University, Changchun 130012, Jilin, China}
\email{shengyh@jlu.edu.cn}

\author{Yanqiu Zhou}
\address{School of Science, Guangxi University of Science and Technology, Liuzhou 545006, China}
\email{1531734482@qq.com}
\date{\today}

\begin{abstract}
In this paper,  we introduce the  cohomology theory of relative Rota-Baxter operators on Leibniz algebras. We use the cohomological approach to study linear and formal deformations of relative Rota-Baxter operators. In particular, the notion of Nijenhuis elements
is introduced to characterize trivial linear deformations.
Formal deformations and extendibility of   order $n$ deformations
of a relative Rota-Baxter operator are also characterized in terms of the cohomology
theory.
\end{abstract}

\subjclass[2010]{17A32, 17B38, 17B62}

\keywords{cohomology, deformation, relative Rota-Baxter operator, Leibniz algebra}

\maketitle

\tableofcontents

\allowdisplaybreaks

\section{Introduction}\mlabel{sec:intr}

The notion of Rota-Baxter operators on associative algebras was introduced in 1960 by G. Baxter \cite{Ba} in his study of
fluctuation theory in probability. Recently it has been found many applications, including in Connes-Kreimer's algebraic approach to the renormalization in perturbative quantum
field theory~\cite{CK}. In the Lie algebra context, a Rota-Baxter
operator of weight zero
was introduced independently in the
1980s as the operator form of the classical Yang-Baxter equation, whereas the classical Yang-Baxter equation plays important roles in many fields in mathematics and mathematical physics such as quantum groups  and integrable systems \cite{CP,STS}. Rota-Baxter operators on super-type algebras were studied in \cite{AMM}, which build relationships between associative superalgebras, Lie superalgebras, L-dendriform superalgebras and pre-Lie superalgebras.
Recently Rota-Baxter operators on Leibniz algebras were studied in \cite{ST}, which is the main ingredient in the study of the twisting theory and the bialgebra theory for Leibniz algebras.
Generally, Rota-Baxter operators can define on operads, which establish a relationship between the splitting of an operad \cite{Bai-Bellier-Guo-Ni,PBG}. For further details on
Rota-Baxter operators, see ~\cite{Gub-AMS,Gub}.

The deformation of algebraic structures began with the seminal
work of Gerstenhaber~\cite{Ge0,Ge} for associative
algebras and followed by its extension to Lie algebras by
Nijenhuis and Richardson~\cite{NR,NR2}. Makhlouf and Silvestrov study the deformation theories of Hom-type algebras \cite{Makhlouf and Silvestrov}. In general, deformation theory was developed
for binary quadratic operads by Balavoine~\cite{Balavoine-1}.
 % For more general operads we refer the reader to the books of Kontsevich-Soibelman~\cite{KSo} and Loday-Vallette~\cite{LV}, and the references therein.
Recently,   the deformation theories of morphisms and $\huaO$-operators were   developed in \cite{ABM,Fregier-Zambon-2,Fregier-Zambon-1,Tang-Bai-Guo-Sheng}. The concept of a Leibniz algebra was introduced by Loday \cite{Loday,Loday and Pirashvili} with the motivation in the study of the periodicity in algebraic K-theory. Recently the structure theories of semisimple Leibniz algebras and Cartan subalgebras of complex finite-dimensional Leibniz algebras were deep studied  in \cite{GKO,Omirov}.    %Leibniz algebras were studied from different aspects due to applications in both mathematics and physics.
%Integration of Leibniz algebras were studied in \cite{BW,Int1} and deformation quantization of Leibniz algebras was studied in \cite{DW}. Leibniz algebras play a key rule in the  Courant algebroids \cite{Roytenberg} and Nambu mechanics \cite{Ta}.
%As the underlying structures of embedding tensors, Leibniz algebras also have applications in higher gauge theories, see \cite{KS,SW} for more details.

In this paper, we study linear deformations and formal deformations of relative Rota-Baxter operators on Leibniz algebras using the cohomological approach. For this purpose, we first define the cohomology of relative Rota-Baxter operators on Leibniz algebras. Then we study linear deformations and introduce the notion of a Nijenhuis element associated to a relative Rota-Baxter operator, which can give rise to trivial linear deformations. We go on studying formal deformations of relative Rota-Baxter operators on Leibniz algebras and show that the infinitesimal of a formal deformation is a 1-cocycle. Finally, we characterize the extendibility of   order $n$ deformations
of a relative Rota-Baxter operator   in terms of the cohomology theory.

The paper is organized as follows. In Section \ref{sec:K}, we introduce the notion of relative Rota-Baxter operators on Leibniz algebras with respect to representations. Given a relative Rota-Baxter operator, there is a natural Leibniz algebra on the representation space. We define the cohomology theory of a relative Rota-Baxter operator on a Leibniz algebra in terms of the cohomology of the Leibniz algebra on the representation space. In Section \ref{sec:def}, we study deformation theory of relative Rota-Baxter operators. Firstly, we study the linear deformation theory of relative Rota-Baxter operators on Leibniz algebras. We introduce the notion of a Nijenhuis element associated to a relative Rota-Baxter operator, which gives rise to a trivial linear deformation of the relative Rota-Baxter operator. We also build a relationship between   linear deformations of relative Rota-Baxter operators and linear deformations of the underlying Leibniz algebras. Secondly, we study the formal deformation theory of relative Rota-Baxter operators. We show that  the infinitesimals of two
equivalent  formal deformations of a  relative Rota-Baxter operator are in the
same first cohomology class of the relative Rota-Baxter operator. Under some condition, we study the rigidity of a relative Rota-Baxter operator. Finally, we study  deformations of order $n$ of a relative Rota-Baxter operator. We show that the obstructions to extension to  deformations of order $n+1$ are given by 2-cocycles.

\vspace{2mm}

\noindent
{\bf Data Availability Statements.} Data sharing is not applicable to this article as no new data were created or analyzed in this study.
\section{Cohomologies of a relative Rota-Baxter operator on a Leibniz algebra }\label{sec:K}

\begin{defi}
A {\bf Leibniz algebra} is a vector space $\g$ together with a bilinear operation $[\cdot,\cdot]_\g:\g\otimes\g\lon\g$ such that
\begin{eqnarray}
\label{Leibniz}[x,[y,z]_\g]_\g=[[x,y]_\g,z]_\g+[y,[x,z]_\g]_\g,\quad\forall x,y,z\in\g.
\end{eqnarray}
\end{defi}

A {\bf representation} of a Leibniz algebra $(\g,[\cdot,\cdot]_{\g})$ is a triple $(V;\rho^L,\rho^R)$, where $V$ is a vector space, $\rho^L,\rho^R:\g\lon\gl(V)$ are linear maps such that the following equalities hold for all $x,y\in\g$,
\begin{eqnarray}
\label{rep-1}\rho^L([x,y]_{\g})&=&[\rho^L(x),\rho^L(y)],\\
\label{rep-2}\rho^R([x,y]_{\g})&=&[\rho^L(x),\rho^R(y)],\\
\label{rep-3}\rho^R(y)\circ \rho^L(x)&=&-\rho^R(y)\circ \rho^R(x).
\end{eqnarray}
Here $[\cdot,\cdot]:\wedge^2\gl(V)\lon\gl(V)$ is the commutator Lie bracket on $\gl(V)$, the vector space of linear transformations on $V$.

Define the left multiplication $L:\g\longrightarrow\gl(\g)$ and the right multiplication $R:\g\longrightarrow\gl(\g)$ by $L_xy=[x,y]_\g$ and $R_xy=[y,x]_\g$ respectively for all $x,y\in \g$.  Then $(\g;L,R)$ is a representation of $(\g,[\cdot,\cdot]_{\g})$, which is called the {\bf regular representation}.

\begin{defi}{\rm (\cite{Loday and Pirashvili})}
Let $(V;\rho^L,\rho^R)$ be a representation of a Leibniz algebra $(\g,[\cdot,\cdot]_{\g})$.
The {\bf Loday-Pirashvili cohomology} of $\g$ with  coefficients in $V$ is the cohomology of the cochain complex $C^k(\g,V)=
\Hom(\otimes^k\g,V), ~(k\ge 0)$ with the coboundary operator
$\partial:C^k(\g,V)\longrightarrow C^
{k+1}(\g,V)$
defined by
\begin{eqnarray*}
(\partial f)(x_1,\cdots,x_{k+1})&=&\sum_{i=1}^{k}(-1)^{i+1}\rho^L(x_i)f(x_1,\cdots,\hat{x_i},\cdots,x_{k+1})+(-1)^{k+1}\rho^R(x_{k+1})f(x_1,\cdots,x_{k})\\
                      \nonumber&&+\sum_{1\le i<j\le k+1}(-1)^if(x_1,\cdots,\hat{x_i},\cdots,x_{j-1},[x_i,x_j]_\g,x_{j+1},\cdots,x_{k+1}),
\end{eqnarray*}
for all $x_1,\cdots, x_{k+1}\in\g$.
\end{defi}

A permutation $\sigma\in\mathbb S_n$ is called an $(i,n-i)$-{\bf shuffle} if $\sigma(1)<\cdots <\sigma(i)$ and $\sigma(i+1)<\cdots <\sigma(n)$. If $i=0$ or $n$ we assume $\sigma=\Id$. The set of all $(i,n-i)$-shuffles will be denoted by $\mathbb S_{(i,n-i)}$. The notion of an $(i_1,\cdots,i_k)$-shuffle and the set $\mathbb S_{(i_1,\cdots,i_k)}$ are defined analogously.

First we introduce the notion of a relative Rota-Baxter operator on a Leibniz algebra and give an example.
\begin{defi}
Let $(V;\rho^L,\rho^R)$ be a representation of a Leibniz algebra $(\g,[\cdot,\cdot]_\g)$. A linear operator $T:V\lon\g$ is called a {\bf relative Rota-Baxter operator} on $(\g,[\cdot,\cdot]_\g)$ with respect to $(V;\rho^L,\rho^R)$ if $T$ satisfies:
\begin{eqnarray}\label{Rota-Baxter}
[Tv_1,Tv_2]_\g=T(\rho^L(Tv_1)v_2+\rho^R(Tv_2)v_1),\,\,\,\,\forall v_1,v_2\in V.
\end{eqnarray}
\end{defi}

\begin{ex}\label{example-5}{\rm
%\yh{Find a 3-dimensional Leibniz algebra and consider the regular representation. Compute symmetric \kup}
Consider the $3$-dimensional Leibniz algebra $(\g,[\cdot,\cdot])$ given with respect to a basis $\{e_1,e_2,e_3\}$  by
$$
[e_1,e_1]=e_3.
$$
 Then $T=\left(\begin{array}{ccc}a_{11}&a_{12}&a_{13}\\
 a_{21}&a_{22}&a_{23}\\
 a_{31}&a_{32}&a_{33}\end{array}\right)$ is a relative Rota-Baxter operator on $(\g,[\cdot,\cdot])$ with respect to the regular representation if and only if
$$
   [Te_i,Te_j]=T([Te_i,e_j]+[e_i,Te_j]),\quad \forall i,j=1,2,3.
   $$
 We have
$
[Te_1,Te_1]=[a_{11}e_1+a_{21}e_2+a_{31}e_3,a_{11}e_1+a_{21}e_2+a_{31}e_3]=a_{11}^2e_3,
$
and
\begin{eqnarray*}
T([Te_1,e_1]+[e_1,Te_1])&=&T([a_{11}e_1+a_{21}e_2+a_{31}e_3,e_1]+[e_1,a_{11}e_1+a_{21}e_2+a_{31}e_3])\\
                        &=&2a_{11}Te_3=2a_{11}a_{13}e_1+2a_{11}a_{23}e_2+2a_{11}a_{33}e_3.
\end{eqnarray*}
Thus, by $[Te_1,Te_1]=T([Te_1,e_1]+[e_1,Te_1])$, we obtain
$$
a_{11}^2=2a_{11}a_{33},\quad a_{11}a_{13}=0,\quad a_{11}a_{23}=0.
$$
Similarly, by $[Te_1,Te_2]=T([Te_1,e_2]+[e_1,Te_2])$, we obtain
$$
a_{11}a_{12}=a_{12}a_{33},\quad a_{12}a_{13}=0,\quad a_{12}a_{23}=0.
$$
By $[Te_1,Te_3]=T([Te_1,e_3]+[e_1,Te_3])$, we obtain
$$
a_{11}a_{13}=a_{13}a_{33},\quad a_{13}a_{13}=0,\quad a_{13}a_{23}=0.
$$
By $[Te_2,Te_1]=T([Te_2,e_1]+[e_2,Te_1])$, we obtain
$$
a_{12}a_{11}=a_{12}a_{33},\quad a_{12}a_{13}=0,\quad a_{12}a_{23}=0.
$$
By $[Te_3,Te_1]=T([Te_3,e_1]+[e_3,Te_1])$, we obtain
$$
a_{13}a_{11}=a_{13}a_{33},\quad a_{13}a_{13}=0,\quad a_{13}a_{23}=0.
$$
By $[Te_i,Te_j]=T([Te_i,e_j]+[e_i,Te_j])=0,~i,j=2,3$, we obtain
$$
a_{12}^2=0,\quad a_{13}^2=0,\quad a_{12}a_{13}=0.
$$

\emptycomment{
\begin{eqnarray*}
[Ke_1,Ke_2]=[a_{11}e_1+a_{21}e_2+a_{31}e_3,a_{12}e_1+a_{22}e_2+a_{32}e_3]=a_{11}a_{12}e_3,
\end{eqnarray*}
and
\begin{eqnarray*}
K([Ke_1,e_2]+[e_1,Ke_2])&=&K[e_1,a_{12}e_1+a_{22}e_2+a_{32}e_3]\\
                        &=&a_{12}Ke_3=a_{12}a_{13}e_1+a_{12}a_{23}e_2+a_{12}a_{33}e_3.
\end{eqnarray*}
----
\begin{eqnarray*}
[Ke_1,Ke_3]=[a_{11}e_1+a_{21}e_2+a_{31}e_3,a_{13}e_1+a_{23}e_2+a_{33}e_3]=a_{11}a_{13}e_3,
\end{eqnarray*}
and
\begin{eqnarray*}
K([Ke_1,e_3]+[e_1,Ke_3])&=&K[e_1,a_{13}e_1+a_{23}e_2+a_{33}e_3]\\
                        &=&a_{13}Ke_3=a_{13}a_{13}e_1+a_{13}a_{23}e_2+a_{13}a_{33}e_3.
\end{eqnarray*}
-----

\begin{eqnarray*}
[Ke_2,Ke_1]=[a_{12}e_1+a_{22}e_2+a_{32}e_3,a_{11}e_1+a_{21}e_2+a_{31}e_3]=a_{12}a_{11}e_3,
\end{eqnarray*}
and
\begin{eqnarray*}
K([Ke_2,e_1]+[e_2,Ke_1])&=&K[a_{12}e_1+a_{22}e_2+a_{32}e_3,e_1]\\
                        &=&a_{12}Ke_3=a_{12}a_{13}e_1+a_{12}a_{23}e_2+a_{12}a_{33}e_3.
\end{eqnarray*}
-----

\begin{eqnarray*}
[Ke_3,Ke_1]=[a_{13}e_1+a_{23}e_2+a_{33}e_3,a_{11}e_1+a_{21}e_2+a_{31}e_3]=a_{13}a_{11}e_3,
\end{eqnarray*}
and
\begin{eqnarray*}
K([Ke_3,e_1]+[e_3,Ke_1])&=&K[a_{13}e_1+a_{23}e_2+a_{33}e_3,e_1]\\
                        &=&a_{13}Ke_3=a_{13}a_{13}e_1+a_{13}a_{23}e_2+a_{13}a_{33}e_3.
\end{eqnarray*}
-------
Moreover, we have
\begin{eqnarray*}
&&[Ke_2,Ke_2]=a_{12}^2e_3,\quad[Ke_2,Ke_3]=a_{12}a_{13}e_3\\
&&[Ke_3,Ke_2]=a_{13}a_{12}e_3,\quad[Ke_3,Ke_3]=a_{13}^2e_3.
\end{eqnarray*}
}
Summarize the above discussion, we have
\begin{itemize}
     \item[\rm(i)] If $a_{11}=a_{12}=a_{13}=0$, then  any $T=\left(\begin{array}{ccc}0&0&0\\
                                                                                    a_{21}&a_{22}&a_{23}\\
                                                                                    a_{31}&a_{32}&a_{33}\end{array}\right)$ is a relative Rota-Baxter operator on $(\g,[\cdot,\cdot])$ with respect to the regular representation.
     \item[\rm(ii)] If $a_{12}=a_{13}=0$ and $a_{11}\not=0,~a_{23}=0$, then  any $T=\left(\begin{array}{ccc}a_{11}&0&0\\
                                                                                    a_{21}&a_{22}&0\\
                                                                                    a_{31}&a_{32}&\half a_{11}\end{array}\right)$ is a relative Rota-Baxter operator on $(\g,[\cdot,\cdot])$ with respect to the regular representation.
   \end{itemize}
   }
\end{ex}

\begin{lem}\label{duble-Leibniz}
Let $T$ be a  relative Rota-Baxter operator on a Leibniz
algebra  $(\g,[\cdot,\cdot]_\g)$ with respect to $(V;\rho^L,\rho^R)$.  Define
\begin{eqnarray}\label{thm:rota-baxter-to-leibniz}
[u,v]_{T}=\rho^L(Tu)v+\rho^R(Tv)u,\quad u,v\in V.
\end{eqnarray}
Then $(V,[\cdot,\cdot]_{T})$ is a Leibniz algebra.
\end{lem}

\begin{proof}
For all $u,v,w\in V$, we have
\begin{eqnarray*}
&&[u,[v,w]_T]_T-[[u,v]_T,w]_T-[v,[u,w]_T]_T\\
&=&[u,\rho^L(Tv)w+\rho^R(Tw)v]_T-[\rho^L(Tu)v+\rho^R(Tv)u,w]_T-[v,\rho^L(Tu)w+\rho^R(Tw)u]_T\\
&=&\rho^L(Tu)\rho^L(Tv)w+\rho^R(T\rho^L(Tv)w)u+\rho^L(Tu)\rho^R(Tw)v+\rho^R(T\rho^R(Tw)v)u\\
&&-\rho^L(T\rho^L(Tu)v)w-\rho^R(Tw)\rho^L(Tu)v-\rho^L(T\rho^R(Tv)u)w-\rho^R(Tw)\rho^R(Tv)u\\
&&-\rho^L(Tv)\rho^L(Tu)w-\rho^R(T\rho^L(Tu)w)v-\rho^L(Tv)\rho^R(Tw)u-\rho^R(T\rho^R(Tw)u)v\\
&\stackrel{\eqref{Rota-Baxter}}{=}&0.
\end{eqnarray*}
Therefore, $(V,[\cdot,\cdot]_{T})$ is a Leibniz algebra.
\end{proof}

\begin{cor}
Let $T$ be a  relative Rota-Baxter operator on a Leibniz
algebra  $(\g,[\cdot,\cdot]_\g)$ with respect to $(V;\rho^L,\rho^R)$. Then   $T$ is a  homomorphism from the Leibniz algebra $(V,[\cdot,\cdot]_{T})$  to the initial Leibniz algebra $(\g,[\cdot,\cdot]_\g)$.
\end{cor}

\begin{proof}
It follows from Lemma \ref{duble-Leibniz} and \eqref{Rota-Baxter}.
\end{proof}

\begin{thm}
Let $T$ be a  relative Rota-Baxter operator on a Leibniz
algebra  $(\g,[\cdot,\cdot]_\g)$ with respect to $(V;\rho^L,\rho^R)$. Define
\begin{eqnarray}\label{rep-Leibniz}
\bar{\rho}^L(u)x=[Tu,x]_\g-T\rho^R(x)u,\quad \bar{\rho}^R(u)x=[x,Tu]_\g-T\rho^L(x)u,\quad\forall u\in V,x\in\g.
\end{eqnarray}
Then $(\g;\bar{\rho}^L,\bar{\rho}^R)$ is a representation of the Leibniz algebra $(V,[\cdot,\cdot]_{T})$.
\end{thm}

\begin{proof}
Indeed, for all $u,v\in V$ and $x\in\g$ , we have
\begin{eqnarray*}
&&(\bar{\rho}^L([u,v]_T)-[\bar{\rho}^L(u),\bar{\rho}^L(v)])x\\
&\stackrel{\eqref{thm:rota-baxter-to-leibniz},\eqref{rep-Leibniz}}{=}&[T[u,v]_T,x]_\g-T\rho^R(x)[u,v]_T-\bar{\rho}^L(u)([Tv,x]_\g-T\rho^R(x)v)+\bar{\rho}^L(v)([Tu,x]_\g-T\rho^R(x)u)\\
&\stackrel{\eqref{Rota-Baxter}}{=}&[[Tu,Tv]_\g,x]_\g-T\rho^R(x)(\rho^L(Tu)v+\rho^R(Tv)u)-[Tu,[Tv,x]_\g-T\rho^R(x)v]_\g\\
&&+T\rho^R([Tv,x]_\g-T\rho^R(x)v)u+[Tv,[Tu,x]_\g-T\rho^R(x)u]_\g-T\rho^R([Tu,x]_\g-T\rho^R(x)u)v\\
&=&-T\rho^R(x)\rho^L(Tu)v-T\rho^R(x)\rho^R(Tv)u+[Tu,T\rho^R(x)v]_\g+T\rho^R([Tv,x]_\g)u\\
&&-T\rho^R(T\rho^R(x)v)u-[Tv,T\rho^R(x)u]_\g-T\rho^R([Tu,x]_\g)v+T\rho^R(T\rho^R(x)u)v\\
&\stackrel{\eqref{rep-2}}{=}&-T\rho^R(x)\rho^L(Tu)v-T\rho^R(x)\rho^R(Tv)u+[Tu,T\rho^R(x)v]_\g+T[\rho^L(Tv),\rho^R(x)]u\\
&&-T\rho^R(T\rho^R(x)v)u-[Tv,T\rho^R(x)u]_\g-T[\rho^L(Tu),\rho^R(x)]v+T\rho^R(T\rho^R(x)u)v\\
&=&-T\rho^R(x)\rho^L(Tu)v-T\rho^R(x)\rho^R(Tv)u+[Tu,T\rho^R(x)v]_\g+T\rho^L(Tv)\rho^R(x)u-T\rho^R(x)\rho^L(Tv)u\\
&&-T\rho^R(T\rho^R(x)v)u-[Tv,T\rho^R(x)u]_\g-T\rho^L(Tu)\rho^R(x)v+T\rho^R(x)\rho^L(Tu)v+T\rho^R(T\rho^R(x)u)v\\
&\stackrel{\eqref{Rota-Baxter}}{=}&0.
\end{eqnarray*}
Thus we deduce that $\bar{\rho}^L([u,v]_T)=[\bar{\rho}^L(u),\bar{\rho}^L(v)]$. Furthermore, for all $u,v\in V$ and $x\in\g$, we have
\begin{eqnarray*}
&&(\bar{\rho}^R([u,v]_T)-[\bar{\rho}^L(u),\bar{\rho}^R(v)])x\\
&\stackrel{\eqref{thm:rota-baxter-to-leibniz},\eqref{rep-Leibniz}}{=}&[x,T[u,v]_T]_\g-T\rho^L(x)[u,v]_T-\bar{\rho}^L(u)([x,Tv]_\g-T\rho^L(x)v)+\bar{\rho}^R(v)([Tu,x]_\g-T\rho^R(x)u)\\
&=&[x,[Tu,Tv]_\g]_\g-T\rho^L(x)(\rho^L(Tu)v+\rho^R(Tv)u)-[Tu,[x,Tv]_\g-T\rho^L(x)v]_\g\\
&&+T\rho^R([x,Tv]_\g-T\rho^L(x)v)u+[[Tu,x]_\g-T\rho^R(x)u,Tv]_\g-T\rho^L([Tu,x]_\g-T\rho^R(x)u)v\\
&=&-T\rho^L(x)\rho^L(Tu)v-T\rho^L(x)\rho^R(Tv)u+[Tu,T\rho^L(x)v]_\g+T\rho^R([x,Tv]_\g)u\\
&&-T\rho^R(T\rho^L(x)v)u-[T\rho^R(x)u,Tv]_\g-T\rho^L([Tu,x]_\g)v+T\rho^L(T\rho^R(x)u)v\\
&\stackrel{\eqref{rep-1},\eqref{rep-2}}{=}&-T\rho^L(x)\rho^L(Tu)v-T\rho^L(x)\rho^R(Tv)u+[Tu,T\rho^L(x)v]_\g+T[\rho^L(x),\rho^R(Tv)]u\\
&&-T\rho^R(T\rho^L(x)v)u-[T\rho^R(x)u,Tv]_\g-T[\rho^L(Tu),\rho^L(x)]v+T\rho^L(T\rho^R(x)u)v\\
&\stackrel{\eqref{rep-3}}{=}&[Tu,T\rho^L(x)v]_\g-T\rho^L(Tu)\rho^L(x)v-T\rho^R(T\rho^L(x)v)u\\
&&-[T\rho^R(x)u,Tv]_\g+T\rho^L(T\rho^R(x)u)v+T\rho^R(Tv)\rho^R(x)u\\
&\stackrel{\eqref{Rota-Baxter}}{=}&0,
\end{eqnarray*}
which implies  that $\bar{\rho}^R([u,v]_T)=[\bar{\rho}^L(u),\bar{\rho}^R(v)]$. Finally, for all $u,v\in V$ and $x\in\g$, we have
\begin{eqnarray*}
&&(\bar{\rho}^R(u)\bar{\rho}^L(v)+\bar{\rho}^R(u)\bar{\rho}^R(v))x\\
&\stackrel{\eqref{rep-Leibniz}}{=}&\bar{\rho}^R(u)([Tv,x]_\g-T\rho^R(x)v)+\bar{\rho}^R(u)([x,Tv]_\g-T\rho^L(x)v)\\
&=&[[Tv,x]_\g-T\rho^R(x)v,Tu]_\g-T\rho^L([Tv,x]_\g-T\rho^R(x)v)u\\
&&+[[x,Tv]_\g-T\rho^L(x)v,Tu]_\g-T\rho^L([x,Tv]_\g-T\rho^L(x)v)u\\
&=&[[Tv,x]_\g,Tu]_\g-[T\rho^R(x)v,Tu]_\g-T\rho^L([Tv,x]_\g)u+T\rho^L(T\rho^R(x)v)u\\
&&+[[x,Tv]_\g,Tu]_\g-[T\rho^L(x)v,Tu]_\g-T\rho^L([x,Tv]_\g)u+T\rho^L(T\rho^L(x)v)u\\
&\stackrel{\eqref{Leibniz},\eqref{rep-1}}{=}&-[T\rho^R(x)v,Tu]_\g-T[\rho^L(Tv),\rho^L(x)]u+T\rho^L(T\rho^R(x)v)u\\
&&-[T\rho^L(x)v,Tu]_\g-T[\rho^L(x),\rho^L(Tv)]u+T\rho^L(T\rho^L(x)v)u\\
&\stackrel{\eqref{Rota-Baxter},\eqref{rep-3}}{=}&0.
\end{eqnarray*}
Therefore,   $(\g;\bar{\rho}^L,\bar{\rho}^R)$ is a representation of the Leibniz algebra $(V,[\cdot,\cdot]_{T})$.
\end{proof}

Let $\partial_T: C^{n}(V,\g)\longrightarrow  C^{n+1}(V,\g)$ be the corresponding Loday-Pirashvili coboundary operator of the Leibniz algebra $(V,[\cdot,\cdot]_{T})$ with  coefficients in the representation $(\g;\bar{\rho}^L,\bar{\rho}^R)$. More precisely, $\partial_T: C^{n}(V,\g)\longrightarrow  C^{n+1}(V,\g)$ is given by
               \begin{eqnarray*}
                && (\partial_T f)(v_1,\cdots,v_{n+1})\\&=&\sum_{i=1}^{n}(-1)^{i+1}[Tv_i,f(v_1,\cdots,\hat{v}_i,\cdots, v_{n+1})]_\g-\sum_{i=1}^{n}(-1)^{i+1}T\rho^R(f(v_1,\cdots,\hat{v}_i,\cdots, v_{n+1}))v_i\\
                &&+(-1)^{n+1}[f(v_1,\cdots,v_{n}),Tv_{n+1}]_\g+(-1)^nT\rho^L(f(v_1,\cdots,v_{n}))v_{n+1}\\
                &&+\sum_{1\le i<j\le n+1}(-1)^{i}f(v_1,\cdots,\hat{v}_i,\cdots,v_{j-1},\rho^L(Tv_i)v_j+\rho^R(Tv_j)v_i,v_{j+1},\cdots, v_{n+1}).
               \end{eqnarray*}

\begin{defi}
Let $T$ be a relative Rota-Baxter operator on a Leibniz algebra $(\g,[\cdot,\cdot]_\g)$ with respect to a representation $(V;\rho^L,\rho^R)$. Denote by $(C^*(V,\g)=\oplus _{k=0}^{+\infty}C^k(V,\g),\partial_T)$ the above cochain complex.  Denote the set of $k$-cocycles by $\huaZ^k(V,\g)$ and the set of $k$-coboundaries by $\huaB^k(V,\g)$. Denote by
  \begin{equation}
  \huaH^k(V,\g)=\huaZ^k(V,\g)/\huaB^k(V,\g), \quad k \geq 0,
  \label{eq:ocoh}
  \end{equation}
the $k$-th cohomology group which will be taken to be the {\bf $k$-th cohomology group for the relative Rota-Baxter operator $T$}.
\label{de:opcoh}
\end{defi}

It is obvious that $x\in\g$ is closed if and only if
         $
           T\circ\rho^L(x)-L_x\circ T=0,
         $
               and   $f\in C^1(V,\g)$  is closed  if and only if
               $$
              [Tu,f(v)]_\g+[f(u),Tv]_\g-T(\rho^L(f(u))v+\rho^R(f(v))u)-f(\rho^L(Tu)v+\rho^R(Tv)u)=0.
               $$

Let $(V;\rho^L,\rho^R)$ be  a representation of a Leibniz algebra $(\g,[\cdot,\cdot]_\g)$. Consider the graded vector space
$$C^*(V,\g):=\oplus_{n\ge 1}C^n(V,\g)=\oplus_{n\geq 1}\Hom(\otimes^{n}V,\g).$$

\begin{thm}{\rm (\cite{ST})}\label{huaO-MC}
With the above notations,   $(C^*(V,\g)=\oplus_{k=1}^{+\infty}C^k(V,\g),\{\cdot,\cdot\})$ is a graded Lie algebra, where the graded Lie  bracket $\{\cdot,\cdot\}$ is given by
{\footnotesize
\begin{eqnarray}
&&\nonumber\{g_1,g_2\}(v_1,v_2,\cdots,v_{m+n})\\
\label{gla-rota-baxter-leibniz}&&=\sum_{k=1}^{m}\sum_{\sigma\in\mathbb S_{(k-1,n)}}(-1)^{(k-1)n+1}(-1)^{\sigma}g_1(v_{\sigma(1)},\cdots,v_{\sigma(k-1)},\rho^L(g_2(v_{\sigma(k)},\cdots,v_{\sigma(k+n-1)}))v_{k+n},v_{k+n+1},\cdots,v_{m+n})\\
\nonumber&&+\sum_{k=2}^{m+1}\sum_{\sigma\in\mathbb S_{(k-2,n,1)}\atop \sigma(k+n-2)=k+n-1}(-1)^{kn}(-1)^{\sigma}
g_1(v_{\sigma(1)},\cdots,v_{\sigma(k-2)},\rho^R(g_2(v_{\sigma(k-1)},\cdots,v_{\sigma(k+n-2)}))v_{\sigma(k+n-1)},v_{k+n},\cdots,v_{m+n})\\
\nonumber&&+\sum_{k=1}^{m}\sum_{\sigma\in\mathbb S_{(k-1,n-1)}}(-1)^{(k-1)n}(-1)^{\sigma}[g_2(v_{\sigma(k)},\cdots,v_{\sigma(k+n-2)},v_{k+n-1}),g_1(v_{\sigma(1)},\cdots,v_{\sigma(k-1)},v_{k+n},\cdots,v_{m+n})]_{\g}\\
\nonumber&&+\sum_{\sigma\in\mathbb S_{(m,n-1)}}(-1)^{mn+1}(-1)^{\sigma}[g_1(v_{\sigma(1)},\cdots,v_{\sigma(m)}),g_2(v_{\sigma(m+1)},\cdots,v_{\sigma(m+n-1)},v_{m+n})]_{\g}\\
\nonumber&&+\sum_{k=1}^{n}\sum_{\sigma\in\mathbb S_{(k-1,m)}}(-1)^{m(k+n-1)}(-1)^{\sigma}g_2(v_{\sigma(1)},\cdots,v_{\sigma(k-1)},\rho^L(g_1(v_{\sigma(k)},\cdots,v_{\sigma(k+m-1)}))v_{k+m},v_{k+m+1},\cdots,v_{m+n})\\
\nonumber&&+\sum_{k=1}^{n}\sum_{\sigma\in\mathbb S_{(k-1,m,1)}\atop\sigma(k+m-1)=k+m}(-1)^{m(k+n-1)+1}(-1)^{\sigma}
g_2(v_{\sigma(1)},\cdots,v_{\sigma(k-1)},\rho^R(g_1(v_{\sigma(k)},\cdots,v_{\sigma(k-1+m)}))v_{\sigma(k+m)},v_{k+m+1},\cdots,v_{m+n}),
\end{eqnarray}
}
for all $g_1\in C^m(V,\g),~g_2\in C^n(V,\g)$.

Moreover, its Maurer-Cartan elements are precisely relative Rota-Baxter operators on the Leibniz algebra $(\g,[\cdot,\cdot]_\g)$ with respect to the representation $(V;\rho^L,\rho^R)$.
\end{thm}

Let $T$ be a relative Rota-Baxter operator on the Leibniz algebra $(\g,[\cdot,\cdot]_\g)$ with respect to the representation $(V;\rho^L,\rho^R)$. Since $T$ is a Maurer-Cartan element of the graded Lie algebra $(C^*(V,\g),\{\cdot,\cdot\})$ given by Theorem~\ref{huaO-MC}, it follows from the graded Jacobi identity that the map
\begin{equation}\label{eq:huaO-dT}
d_{T}:C^n(V,\g)\longrightarrow C^{n+1}(V,\g), \quad d_{T}f=\{T,f\},
\end{equation}
 is a graded derivation of the graded Lie algebra $(C^*(V,\g),\{\cdot,\cdot\})$ satisfying $d_{T}\circ d_{T}=0$.

 Up to a sign, the coboundary operators $\partial_T$ coincides with the differential operator $d_{T}$ defined by ~\eqref{eq:huaO-dT} using  the Maurer-Cartan element $T$.

 \begin{thm}\label{partial-to-derivation}
 Let $T$ be a relative Rota-Baxter operator on the Leibniz algebra $(\g,[\cdot,\cdot]_\g)$ with respect to the representation $(V;\rho^L,\rho^R)$. Then we have
 $$
{\partial}_T f=(-1)^{n-1}d_Tf,\quad \forall f\in \Hom(\otimes^nV,\g),~n=1,2,\cdots.
 $$
\end{thm}

\begin{proof}
For all $v_1,v_2,\cdots,v_{n+1}\in V$ and $f\in \Hom(\otimes^nV,\g)$, we have
\begin{eqnarray*}
&&(d_Tf)(v_1,v_2,\cdots,v_{n+1})\\&=&\{T,f\}(v_1,v_2,\cdots,v_{n+1})\\
                              &=&(-1)^1T(\rho^L(f(v_1,\cdots,v_n))v_{n+1})+\sum_{\sigma\in\mathbb S_{(0,n,1)}\atop\sigma(n)=n+1}(-1)^{\sigma}
                              T(\rho^R(v_{\sigma(1)},\cdots,v_{\sigma(n)})v_{\sigma(n+1)})\\
                              &&+[f(v_1,\cdots,v_n),Tv_{n+1}]_\g+\sum_{\sigma\in\mathbb S_{(1,n-1)}}(-1)^{n+1}(-1)^{\sigma}[Tv_{\sigma(1)},f(v_{\sigma(2)},\cdots,v_{\sigma(n)},v_{n+1})]_\g\\
                              &&+\sum_{k=1}^{n}\sum_{\sigma\in\mathbb S_{(k-1,1)}}(-1)^{k+n-1}(-1)^{\sigma}f(v_{\sigma(1)},\cdots,v_{\sigma(k-1)},\rho^L(Tv_{\sigma(k)})v_{k+1},v_{k+2},\cdots,v_{n+1})\\
                              &&+\sum_{k=1}^{n}\sum_{\sigma\in\mathbb S_{(k-1,1,1)}\atop\sigma(k)=k+1}(-1)^{k+n}(-1)^{\sigma}f(v_{\sigma(1)},\cdots,v_{\sigma(k-1)},\rho^R(Tv_{\sigma(k)})v_{\sigma(k+1)},v_{k+2},\cdots,v_{n+1})\\
                              &=&(-1)^1T(\rho^L(f(v_1,\cdots,v_n))v_{n+1})+\sum_{i=1}^{n}(-1)^{n+1-i}
                              T(\rho^R(v_{1},\cdots,\hat{v}_{i},\cdots,v_n,v_{n+1})v_{i})\\
                              &&+[f(v_1,\cdots,v_n),Tv_{n+1}]_\g+\sum_{i=1}^{n}(-1)^{n+i}[Tv_{i},f(v_{1},\cdots,\hat{v}_{i},\cdots,v_n,v_{n+1})]_\g\\
                              &&+\sum_{1\le i\le j\le n+1}(-1)^{n-1-i}f(v_{1},\cdots,\hat{v}_{i},\cdots,v_{j-1},\rho^L(Tv_{i})v_{j},v_{j+1},\cdots,v_{n+1})\\
                              &&+\sum_{1\le i\le j\le n+1}(-1)^{n-1-i}f(v_{1},\cdots,\hat{v}_{i},\cdots,v_{j-1},\rho^R(Tv_{j})v_{i},v_{j+1},\cdots,v_{n+1}).
\end{eqnarray*}
Thus, we obtain that $\partial_T f=(-1)^{n-1}d_Tf$. The proof is finished.
\end{proof}

We can use these cohomology groups to characterize linear and formal deformations of
relative Rota-Baxter operators in the following section.
\section{Deformations of a  relative Rota-Baxter operator}
\label{sec:def}
\subsection{Linear  deformations of a  relative Rota-Baxter operator}
In  this subsection, we study linear  deformations of a  relative Rota-Baxter operator using the cohomology  theory  given in the
previous section. In particular, we introduce the notion of a
Nijenhuis element associated to a relative Rota-Baxter operator, which gives
rise to a trivial linear deformation of the
relative Rota-Baxter operator.  %\yh{We replace `pre-Lie' by Leibniz (not Leibniz-dendriform! we do not use Leibniz-dendriform algebra in this paper). There is an immediate question: Is linear deformations (Nijenhuis operator) of Leibniz algebras studied in the literature? If yes, we cite and establish the relation. If not, we add this simple part and establish the relation. Please check. }

    \begin{defi}
      Let $T$ and $T'$ be relative Rota-Baxter operators on a Leibniz algebra $(\g,[\cdot,\cdot]_\g)$ with respect to a representation $(V;\rho^L,\rho^R)$. A {\bf homomorphism} from $T'$ to $T$ consists of a Leibniz algebra homomorphism  $\phi_\g:\g\longrightarrow\g$ and a linear map $\phi_V:V\longrightarrow V$ such that
      \begin{eqnarray}
        T\circ \phi_V&=&\phi_\g\circ T',\mlabel{defi:isocon1}\\
        \phi_V\rho^L(x)u&=&\rho^L(\phi_\g(x))\phi_V(u)\mlabel{defi:isocon2}\\
        \phi_V\rho^R(x)u&=&\rho^R(\phi_\g(x))\phi_V(u),\quad\forall x\in\g, u\in V.\mlabel{defi:isocon3}
      \end{eqnarray}
      In particular, if both $\phi_\g$ and $\phi_V$ are  invertible,  $(\phi_\g,\phi_V)$ is called an  {\bf isomorphism}  from $T'$ to $T$.
    \mlabel{defi:isoO}
    \end{defi}

\begin{pro}
Let $T$ and $T'$ be two relative Rota-Baxter operators on a Leibniz algebra $(\g,[\cdot,\cdot]_\g)$ with respect to a representation $(V;\rho^L,\rho^R)$ and $(\phi_\g,\phi_V)$   a homomorphism (resp. an isomorphism) from $T'$ to $T$. Then $\phi_V$ is a homomorphism (resp. an isomorphism) of Leibniz algebras from $(V,[\cdot,\cdot]_{T'})$ to $(V,[\cdot,\cdot]_{T})$. %\yh{change pre-Lie to Leibniz}
\end{pro}

\begin{proof}
 This is because, for all $u,v\in V$, we
have
\begin{eqnarray*}
\phi_V([u,v]_{T'})=\phi_V(\rho^L(T'u)v+\rho^R(T'v)u)
&=&\rho^L(\phi_\g(T'u))\phi_V(v)+\rho^R(\phi_\g(T'v))\phi_V(u)\\
&=&\rho^L(T\phi_V(u))\phi_V(v)+\rho^R(T\phi_V(v))\phi_V(u)\\
&=&[\phi_V(u),\phi_V(v)]_T.
\end{eqnarray*}
\end{proof}

    \begin{defi}
    Let $T$   be a relative Rota-Baxter operator  on a Leibniz algebra $(\g,[\cdot,\cdot]_\g)$ with respect to a representation $(V;\rho^L,\rho^R)$ and $\frkT:V\longrightarrow\g$ a linear map. If  $T_t=T+t\frkT$ is still a relative Rota-Baxter operator  on the Leibniz algebra $\g$ with respect to the representation $(V;\rho^L,\rho^R)$ for all $t\in\K$, we say that $\frkT$ generates a {\bf linear deformation} of the relative Rota-Baxter operator $T$.
    \end{defi}
It is direct to check that $T_t=T+t\frkT$ is a linear deformation of a relative Rota-Baxter operator $T$ if and only if
 for any $u,v\in V$,
\vspace{-.1cm}
\begin{eqnarray}
~[Tu,\frkT v]_\g+[\frkT u,Tv]_\g&=&T(\rho^L(\frkT u)v+\rho^R(\frkT v)u)+\frkT(\rho^L(Tu)v+\rho^R(Tv)u),\mlabel{eq:deform1}\\
~[\frkT u,\frkT v]_\g&=&\frkT(\rho^L(\frkT u)v+\rho^R(\frkT v)u).
\mlabel{eq:deform2}
\vspace{-.1cm}
\end{eqnarray}
Note that Eq.~\eqref{eq:deform1} means that $\frkT$ is a 1-cocycle of the Leibniz algebra $(V,[\cdot,\cdot]_T)$ with coefficients in the representation $(\g;\bar{\rho}^L,\bar{\rho}^R)$ and Eq.~\eqref{eq:deform2} means that $\frkT$ is a relative Rota-Baxter operator on the Leibniz algebra $\g$ with respect to the representation $(V;\rho^L,\rho^R)$.

 Let $(V,[\cdot,\cdot])$ be a Leibniz algebra and $\omega:\otimes^2V\longrightarrow V$ be a linear map. If for any $t\in\K$, the multiplication $[\cdot,\cdot]_t$ defined by
\vspace{-.3cm}
$$
[u,v]_t=[u,v]+t\omega(u,v), \;\forall u,v\in V,
\vspace{-.2cm}
$$
also gives a Leibniz algebra structure, we say that $\omega$ generates a {\bf linear deformation} of the Leibniz algebra $(V,[\cdot,\cdot])$.

The two types of linear deformations are related as follows.

\begin{pro}
 If $\frkT$ generates a linear deformation of a relative Rota-Baxter operator $T$ on a Leibniz algebra $(\g,[\cdot,\cdot]_\g)$ with respect to a representation $(V;\rho^L,\rho^R)$, then the product $\omega_\frkT$ on $V$ defined by
\vspace{-.2cm}
   $$
   \omega_\frkT(u,v)=\rho^L(\frkT u)v+\rho^R(\frkT v)u,\quad\forall u,v\in V,
   $$
generates a linear deformation of the associated Leibniz algebra $(V,[\cdot,\cdot]_{T})$.
\end{pro}

\begin{proof}
We denote by $[\cdot,\cdot]_t$ the corresponding Leibniz algebra structure associated to the relative Rota-Baxter operator $T+t\frkT$. Then we have
\begin{eqnarray*}
[u,v]_t&=&\rho^L((T+t\frkT)u)v+\rho^R((T+t\frkT)v)u\\
&=&\rho^L(Tu)v+\rho^R(Tv)u+t(\rho^L(\frkT
u)v+\rho^R(\frkT v)u)\\
&=&[u,v]_T+t\omega_\frkT (u,v), \quad \forall u,v\in V,
\end{eqnarray*}
which implies that $\omega_\frkT$ generates a  linear deformation of $(V,[\cdot,\cdot]_{T})$.\end{proof}

\begin{defi} Let $T$ be a relative Rota-Baxter operator on a Leibniz algebra $(\g,[\cdot,\cdot]_\g)$
with respect to a representation $(V;\rho^L,\rho^R)$. Two
linear deformations $T^1_t=T+t\frkT_1$ and
$T^2_t=T+t\frkT_2$ are said to be {\bf equivalent} if there exists
an $x\in\g$ such that $({\Id}_\g+tL_x,{\Id}_V+t\rho^L(x))$ is a
homomorphism   from $T^2_t$ to $T^1_t$. In particular, a
 linear deformation $T_t=T+t\frkT$ of a
relative Rota-Baxter operator $T$ is said to be {\bf trivial} if there exists
an $x\in\g$ such that $({\Id}_\g+tL_x,{\Id}_V+t\rho^L(x))$ is a
homomorphism   from $T_t$ to $T$.
\end{defi}

Let $({\Id}_\g+tL_x,{\Id}_V+t\rho^L(x))$ be a homomorphism from
$T^2_t$ to $T^1_t$. Then ${\Id}_\g+tL_x$ is a Leibniz algebra
endomorphism of $(\g,[\cdot,\cdot]_\g)$. Thus, we have
$$
({\Id}_\g+tL_x)[y,z]_\g=[({\Id}_\g+tL_x)(y),({\Id}_\g+tL_x)(z)]_\g, \;\forall y,z\in \g,
$$
which implies that $x$ satisfies
\begin{equation}
 [[x,y]_\g,[x,z]_\g]_\g=0,\quad \forall y,z\in\g.
 \mlabel{eq:Nij1}
 \end{equation}
Then by Eq.~\eqref{defi:isocon1}, we get
$$
(T+t\frkT_1)({\Id}_V+t\rho^L(x))(u)=({\Id}_\g+tL_x)(T+t\frkT_2)(u),\quad\forall u\in V,
$$
which implies
\begin{eqnarray}
 (\frkT_2-\frkT_1)(u)&=&T\rho^L(x)u-[x,Tu]_\g,\mlabel{eq:deforiso1} \\
  \frkT_1\rho^L(x)(u)&=&[x,\frkT_2u]_\g, \; \forall u\in V.
  \mlabel{eq:deforiso2}
\end{eqnarray}
By Eq.~\eqref{defi:isocon2}, we obtain
$$
({\Id}_V+t\rho^L(x))\rho^L(y)(u)=\rho^L(({\Id}_\g+tL_x)(y))({\Id}_V+t\rho^L(x))(u),\quad \forall y\in\g, u\in V,
$$
which implies that $x$ satisfies
\begin{equation}
  \rho^L([x,y]_\g)\rho^L(x)=0,\quad\forall y\in\g.\mlabel{eq:Nij2}
\end{equation}
Finally, Eq.~\eqref{defi:isocon3} gives
$$
({\Id}_V+t\rho^L(x))\rho^R(y)(u)=\rho^R(({\Id}_\g+tL_x)(y))({\Id}_V+t\rho^L(x))(u),\quad \forall y\in\g, u\in V,
$$
which implies that $x$ satisfies
\begin{equation}
  \rho^R([x,y]_\g)\rho^L(x)=0,\quad\forall y\in\g.\mlabel{eq:Nij3}
\end{equation}
Note that Eq.~\eqref{eq:deforiso1} means that $\frkT_2-\frkT_1=\partial_T x$. Thus, we have

\begin{thm}\label{thm:iso3} Let $T$ be a  relative Rota-Baxter operator  on a Leibniz algebra $(\g,[\cdot,\cdot]_\g)$
with respect to a representation $(V;\rho^L,\rho^R)$.
  If two linear deformations $T^1_t=T+t\frkT_1$ and $T^2_t=T+t\frkT_2$ are equivalent, then $\frkT_1$ and $\frkT_2$ are in the same cohomology class of $ \huaH^1(V,\g)= \huaZ^1(V,\g)/ \huaB^1(V,\g)$ defined in
  Definition~\ref{de:opcoh}.
\end{thm}

    \begin{defi}
Let $T$ be a  relative Rota-Baxter operator on a Leibniz algebra $(\g,[\cdot,\cdot]_\g)$ with respect to a representation $(V;\rho^L,\rho^R)$. An element $x\in\g$ is called a {\bf Nijenhuis element} associated to $T$ if $x$ satisfies Eqs.~\eqref{eq:Nij1}, \eqref{eq:Nij2}, \eqref{eq:Nij3} and the equation
      \begin{eqnarray}
        ~[x,T\rho^L(x)u-[x,Tu]_\g]_\g=0,\quad \forall u\in V.
         \mlabel{eq:Nijenhuis}
        \end{eqnarray}
   Denote by $\Nij(T)$ the set of Nijenhuis elements associated to a relative Rota-Baxter operator $T$.
    \end{defi}

By Eqs.~\eqref{eq:Nij1}-\eqref{eq:Nij3}, it is obvious that a
trivial linear deformation gives rise to a
Nijenhuis element. Conversely, a Nijenhuis element can also generate a trivial linear deformation as the following theorem shows.

  \begin{thm}\label{thm:trivial}
   Let $T$ be a relative Rota-Baxter operator on a Leibniz algebra $(\g,[\cdot,\cdot]_\g)$ with respect to a representation $(V;\rho^L,\rho^R)$. Then for any  $x\in \Nij(T)$, $T_t=T+t \frkT$ with $\frkT=\partial_T x$ is a trivial linear  deformation of the relative Rota-Baxter operator $T$.
\end{thm}
We need the following lemma to prove this theorem.

\begin{lem}\label{lem:isomorphism}
Let $T$ be a relative Rota-Baxter operator on a Leibniz algebra $(\g,[\cdot,\cdot]_\g)$ with respect to a representation $(V;\rho^L,\rho^R)$.  Let $\phi_\g:\g\longrightarrow\g$ be a Leibniz algebra isomorphism and $\phi_V:V\longrightarrow V$ an isomorphism of vector spaces such that Eqs.~\eqref{defi:isocon2}-\eqref{defi:isocon3} hold. Then $\phi_\g^{-1}\circ{T}\circ\phi_V$ is a relative Rota-Baxter operator on the Leibniz algebra $(\g,[\cdot,\cdot]_\g)$ with respect to the representation $(V;\rho^L,\rho^R)$.
\end{lem}

\begin{proof}
 It follows from
straightforward computations.
\end{proof}

{\bf The proof of Theorem \ref{thm:trivial}:}
For any Nijenhuis element $x\in\Nij({T})$, we define
 \begin{eqnarray}\label{trivial-genetator}
\frkT=\partial_T x.
 \end{eqnarray}
By the definition of Nijenhuis elements of ${T}$, for any $t$, $T_t=T+t \frkT$ satisfies
\begin{eqnarray*}
      ({\Id}_\g+tL_x)\circ T_t&=&T\circ({\Id}_V+t\rho^L(x)),\\
       ({\Id}_V+t\rho^L(x))\rho^L(y)u&=&\rho^L( ({\Id}_\g+tL_x)y)({\Id}_V+t\rho^L(x))(u)\\
        ({\Id}_V+t\rho^L(x))\rho^R(y)u&=&\rho^R( ({\Id}_\g+tL_x)y)({\Id}_V+t\rho^L(x))(u),\quad\forall y\in\g, u\in V
\end{eqnarray*}
 For $t$ sufficiently small, we see that ${\Id}_\g+tL_x$ is a Leibniz algebra isomorphism and ${\Id}_V+t\rho^L(x)$ is an isomorphism of vector spaces. Thus, we have
$$T_t=({\Id}_\g+tL_x)^{-1}\circ{T}\circ ({\Id}_V+t\rho^L(x)).$$
By Lemma \ref{lem:isomorphism},
we deduce that ${T}_t$ is a relative Rota-Baxter operator on the Leibniz algebra $(\g,[\cdot,\cdot]_\g)$ with respect to the representation $(V;\rho^L,\rho^R)$, for $t$ sufficiently small.
  Thus, $\frkT$
  given by
Eq.~\eqref{trivial-genetator} satisfies the conditions
\eqref{eq:deform1} and \eqref{eq:deform2}. Therefore,
${T}_t$ is a relative Rota-Baxter operator for all
$t$, which means that $\frkT$
given by
Eq.~\eqref{trivial-genetator} generates a linear deformation of $T$. It is straightforward to see  that
this linear deformation is trivial.\qed

\vspace{2mm}
Now we introduce the notion of a Nijenhuis operator on a Leibniz algebra, which gives rise to a trivial linear deformation of a Leibniz algebra.

\begin{defi}
  A linear map $N:\g\longrightarrow \g$ on a Leibniz algebra $(\g,[\cdot,\cdot]_\g) $ is called a {\bf Nijenhuis operator} if
  \begin{equation}
   [Nx,Ny]_\g=N([Nx,y]_\g+[x,Ny]_\g-N[x,y]_\g),\quad x,y\in\g.
  \end{equation}
\end{defi}

For its connection with a Nijenhuis element associated to
a relative Rota-Baxter operator, we have

\begin{pro}
Let $x\in\g$ be a Nijenhuis element associated to a
relative Rota-Baxter operator $T$ on a Leibniz algebra  $(\g,[\cdot,\cdot]_\g)$ with respect to a
representation $(V;\rho^L,\rho^R)$. Then $\rho^L(x)$ is a Nijenhuis operator
on the associated Leibniz algebra $(V,[\cdot,\cdot]_{T})$.
\end{pro}

\begin{proof}
For all $u,v\in V$, we have
    \begin{eqnarray*}
    &&\rho^L(x)([\rho^L(x)u,v]_T+[u,\rho^L(x)v]_T-\rho^L(x)[u,v]_T)-[\rho^L(x)u,\rho^L(x)v]_T\\
&\stackrel{\eqref{thm:rota-baxter-to-leibniz}}{=}&\rho^L(x)\Big(\rho^L(T\rho^L(x)u)v+\rho^R(Tv)\rho^L(x)u+\rho^L(Tu)\rho^L(x)v+\rho^R(T\rho^L(x)v)u-\rho^L(x)(\rho^L(Tu)v+\rho^R(Tv)u)\Big)\\
&&-\rho^L(T\rho^L(x)u)\rho^L(x)v-\rho^R(T\rho^L(x)v)\rho^L(x)u\\
&=&[\rho^L(x),\rho^L(T\rho^L(x)u)]v+[\rho^L(x),\rho^R(T\rho^L(x)v)]u+\rho^L(x)([\rho^R(Tv),\rho^L(x)]u+[\rho^L(Tu),\rho^L(x)]v)\\
&\stackrel{\eqref{rep-1},\eqref{rep-2}}{=}&\rho^L([x,T\rho^L(x)u]_\g)v+\rho^R([x,T\rho^L(x)v]_\g)u-\rho^L(x)(\rho^R([x,Tv]_\g)u+\rho^L([x,Tu]_\g)v)\\
&\stackrel{\eqref{eq:Nij2},\eqref{eq:Nij3}}{=}&\rho^L([x,T\rho^L(x)u]_\g)v+\rho^R([x,T\rho^L(x)v]_\g)u-[\rho^L(x),\rho^R([x,Tv]_\g)]u-[\rho^L(x),\rho^L([x,Tu]_\g)]v\\
&\stackrel{\eqref{rep-1},\eqref{rep-2}}{=}&\rho^L([x,T\rho^L(x)u]_\g)v+\rho^R([x,T\rho^L(x)v]_\g)u-\rho^R([x,[x,Tv]_\g]_\g)u-\rho^L([x,[x,Tu]_\g]_\g)v\\
&=&\rho^L([x,T\rho^L(x)u-[x,Tu]_\g]_\g)v+\rho^R([x,T\rho^L(x)v-[x,Tv]_\g]_\g)u\\
&\stackrel{\eqref{eq:Nijenhuis}}{=}&0.
    \end{eqnarray*}
Thus, we deduce that $\rho^L(x)$ is a Nijenhuis operator
on the Leibniz algebra $(V,[\cdot,\cdot]_{T})$.
\end{proof}

\begin{rmk}
  When the Leibniz algebra $(\g,[\cdot,\cdot]_\g)$ is a Lie algebra  and $\rho^R=-\rho^L$, we recover the notion of a Nijenhuis element associated to an $\huaO$-operator on a Lie algebra with respect to a representation. See \cite{Tang-Bai-Guo-Sheng} for more details.
\end{rmk}
\subsection{Formal deformations of a relative Rota-Baxter operator}
\mlabel{sec:fordef}

Let $\K[[t]]$ be the ring of power series in one variable
$t$. For any $\K$-linear space $V$, we let
$V[[t]]$ denote the vector space of formal power series in $t$ with coefficients in $V$. If in addition, $(\g,[\cdot,\cdot]_\g)$ is a Leibniz algebra over $\K$, then there is a $\K[[t]]$-Leibniz algebra structure on
$\g[[t]]$ given by
\begin{equation}\label{formal-1}
\bigg[\sum_{i=0}^{+\infty} x_it^i,\sum_{j=0}^{+\infty}y_jt^j\bigg]_\g=\sum_{k=
0}^{+\infty}\sum_{i+j=k}[x_i,y_j]_\g t^k,\quad\forall  x_i,y_j\in \g.
\end{equation}
For any representation $(V;\rho^L,\rho^R)$ of $(\g,[\cdot,\cdot]_\g)$, there is a natural
representation of $\K[[t]]$-Leibniz algebra $\g[[t]]$ on $\K[[t]]$-module $V[[t]]$, which is given by
\begin{eqnarray}
\label{formal-2}\rho^L\bigg(\sum_{i=0}^{+\infty} x_it^i\bigg)\bigg(\sum_{j=0}^{+\infty} v_jt^j\bigg)&=&\sum_{k=
0}^{+\infty}\sum_{i+j=k}\rho^L(x_i)v_j t^k,\\
\label{formal-3}\rho^R\bigg(\sum_{i=0}^{+\infty} x_it^i\bigg)\bigg(\sum_{j=0}^{+\infty} v_jt^j\bigg)&=&\sum_{k=
0}^{+\infty}\sum_{i+j=k}\rho^R(x_i)v_j t^k\quad\forall  x_i\in \g,~v_j\in V.
\end{eqnarray}

Let $T$ be a relative Rota-Baxter operator $T$ on a Leibniz algebra  $(\g,[\cdot,\cdot]_\g)$ with respect to a
representation $(V;\rho^L,\rho^R)$.
Consider a power series
\begin{eqnarray}
T_t=\sum_{i=0}^{+\infty}\frkT_i t^i,\quad \frkT_i\in \Hom_{\K}(V,\g),
\label{eq:tdeform}
\end{eqnarray}
that is, $T_t\in{\rm Hom}_{\K}(V,\g)[[t]]={\rm
Hom}_{\K}(V,\g[[t]])$.  Extend it to be a $\K[[t]]$-module map from $V[[t]]$ to $\g[[t]]$ which is still denoted by $T_t$.

\begin{defi}\label{defi:dO}
 If  $T_t=\sum_{i=0}^{+\infty}\frkT_i t^i$ with $\frkT_0=T$ satisfies
\begin{eqnarray}
[T_t(u),T_t(v)]_\g=T_t\Big(\rho^L(T_t(u))v+\rho^R(T_t(v))u\Big),\;\;\forall
u,v\in V, \mlabel{O-operator}
\end{eqnarray}
we say that $T_t$ is a {\bf formal deformation} of
the relative Rota-Baxter operator $T$.
\end{defi}

Recall  that a formal deformation of a
Leibniz algebra $(\g,[\cdot,\cdot]_\g)$ is a power series
$\omega_t=\sum_{i=0}^{+\infty} \omega_i t^i$ such that $\omega_0(x,y)=[x,y]_\g$
 for any $x,y\in \g$  and $\omega_t$ defines a $\K[[t]]$-Leibniz algebra
structure on $\g[[t]]$.

Building on the relationship between relative Rota-Baxter operators and Leibniz algebras, we have
\begin{pro}
 If $T_t=\sum_{i=0}^{+\infty}\frkT_i t^i$ is a formal deformation of a relative Rota-Baxter operator $T$ on a Leibniz algebra $(\g,[\cdot,\cdot]_\g)$ with respect to a representation $(V;\rho^L,\rho^R)$, then $[\cdot,\cdot]_{T_t}$ defined by
   $$
   [u,v]_{T_t}=\sum_{i=0}^{+\infty}\Big(\rho^L(\frkT_iu)v+\rho^R(\frkT_iv)u\Big) t^i,\quad\forall u,v\in V,
   $$
is a formal deformation of the associated Leibniz algebra $(V,[\cdot,\cdot]_{T})$.
\end{pro}

Applying Eqs. \eqref{formal-1}-\eqref{eq:tdeform} to expand Eq.~\eqref{O-operator} and collecting coefficients of $t^n$, we see that Eq.~\eqref{O-operator} is equivalent to the system of equations
\begin{eqnarray}
\sum\limits_{i+j=k\atop
i,j\geq0}\Big([\frkT_iu,\frkT_jv]_\g-\frkT_i\big(\rho^L(\frkT_ju)v+\rho^R(\frkT_jv)u\big)\Big)=0,
\;\;\forall k\geq 0, u,v\in V. \mlabel{deformation-equation}
\end{eqnarray}

\begin{pro}
Let $T_t=\sum_{i=0}^{+\infty}\frkT_i t^i$ be a formal
deformation of a relative Rota-Baxter operator $T$ on a Leibniz algebra $(\g,[\cdot,\cdot]_\g)$ with respect to a representation $(V;\rho^L,\rho^R)$. Then $\frkT_1$ is a
$1$-cocycle for the relative Rota-Baxter operator $T$, that is, $\partial_T \frkT_1=0$.
\mlabel{pro:cocycle}
\end{pro}
\begin{proof}
For $k=1$, Eq.~\eqref{deformation-equation} is equivalent to
  \vspace{-.1cm}
$$ [Tu,\frkT_1v]+[\frkT_1u,Tv]-T(\rho^L(\frkT_1u)v+\rho^R(\frkT_1v)u)-\frkT_1(\rho^L(Tu)v+\rho^R(Tv)u)=0,\;\;\forall u,v\in V.
  \vspace{-.1cm}
$$
Thus, $\frkT_1$ is a $1$-cocycle. The proof is finished.
\end{proof}

\begin{defi}
Let $T$ be a relative Rota-Baxter operator on a Leibniz algebra $(\g,[\cdot,\cdot]_\g)$ with respect to a representation $(V;\rho^L,\rho^R)$. The $1$-cocycle $\frkT_1$ given in
Proposition~\ref{pro:cocycle} is called the {\bf infinitesimal} of
the formal deformation $T_t=\sum_{i=0}^{+\infty}\frkT_i t^i$ of $T$.
  \vspace{-.1cm}
\end{defi}

\begin{defi}
 Two formal deformations $\overline{T}_t =\sum_{i=0}^{+\infty}\bar{\frkT}_i t^i$ and $T_t=\sum_{i=0}^{+\infty}\frkT_i
 t^i$ of a relative Rota-Baxter operator $T={\bar {\frkT}}_0=\frkT_0$
on a Leibniz algebra $(\g,[\cdot,\cdot]_\g)$ with respect to a representation $(V;\rho^L,\rho^R)$
are said to be {\bf equivalent} if there exist  $x\in\g$,
$\phi_i\in\gl(\g)$ and $\varphi_i\in\gl(V)$, $i\geq 2$, such that
for
\begin{equation}
\phi_t={\Id}_\g+tL_x+\sum_{i=2}^{+\infty}\phi_it^i,\;\;\varphi_t={\Id}_V+t\rho^L(x)+\sum_{i=2}^{+\infty}\varphi_it^i,
\label{eq:phi5}
\vspace{-.4cm}
\end{equation}
the following conditions hold:
\begin{enumerate}
\item[\rm(i)] $[\phi_t(x),\phi_t(y)]_\g=\phi_t[x,y]_\g$ for all
$x,y\in \g;$
\item[\rm(ii)]
$\varphi_t\rho^L(x)u=\rho^L(\phi_t(x))\varphi_t(u)$;
\item[\rm(iii)]
$\varphi_t\rho^R(x)u=\rho^R(\phi_t(x))\varphi_t(u)$ for all $x\in\g,
u\in V$;
\item[\rm(iv)]   $T_t\circ
\varphi_t=\phi_t\circ \overline{T}_t$ as $\K[[t]]$-module maps.
\end{enumerate}
In particular, a formal deformation $T_t$ of a relative Rota-Baxter operator
$T$ is said to be {\bf trivial} if there exists an $x\in\g$,
$\phi_i\in\gl(\g)$ and $\varphi_i\in\gl(V)$, $i\geq 2$, such that
$(\phi_t,\varphi_t)$ defined by Eq.~(\ref{eq:phi5}) gives an
equivalence  between $T_t$ and $T$, with the latter regarded as a
deformation of itself.
\end{defi}

\begin{thm}
If two formal deformations of a relative Rota-Baxter operator
on a Leibniz algebra $(\g,[\cdot,\cdot]_\g)$ with respect to a representation $(V;\rho^L,\rho^R)$ are equivalent, then their infinitesimals are in the
same cohomology class.
\end{thm}

\begin{proof} Let $(\phi_t,\varphi_t)$ be the two maps defined by
Eq.~(\ref{eq:phi5}) which gives an equivalence between two
deformations $\overline{T}_t=\sum_{i=0}^{+\infty}\overline{\frkT}_i t^i$ and
$T_t=\sum_{i=0}^{+\infty}\frkT_i t^i$ of a relative Rota-Baxter operator $T$. By $\phi_t \circ\overline{T}_t=
T_t\circ\varphi_t$, we have
\begin{eqnarray*}
\bar{\frkT}_1v=\frkT_1v+T\rho^L(x)v-[x,Tv]_\g=\frkT_1 v+(\partial_T
x)(v),\quad \forall v\in V,
\end{eqnarray*}
which implies that $\bar{\frkT}_1$ and $\frkT_1$ are in the same
cohomology class.   \end{proof}

\begin{defi}
A relative Rota-Baxter operator $T$ is {\bf rigid} if all formal deformations of $T$ are trivial.
\end{defi}

As a cohomological condition of the rigidity, we have the following result which shows that the rigidity of a relative Rota-Baxter operator is a very strong condition.

\begin{pro} Let $T$ be a relative Rota-Baxter operator on a Leibniz algebra $(\g,[\cdot,\cdot]_\g)$ with respect to a representation $(V;\rho^L,\rho^R)$. If $\huaZ^1(V,\g)=\partial_T(\Nij(T))$, then  $T$ is rigid.
\end{pro}
\begin{proof} Let
$T_t=\sum_{i=0}^{+\infty}\frkT_i t^i$ be a formal
deformation of the relative Rota-Baxter operator $T$. By Proposition
\ref{pro:cocycle}, we deduce $\frkT_1\in \huaZ^1(V,\g)$. By the
assumption $\huaZ^1(V,\g)=\partial_T(\Nij(T))$,  we obtain $\frkT_1=-\partial_T x$ for some $x\in\Nij(T)$. Then
setting $\phi_t={\Id_\g}+tL_x$ and $\varphi_t={\Id}_V+t\rho^L(x)$,
we get a formal deformation $\overline{T}_t:=\phi_t^{-1}\circ
T_t\circ\varphi_t.$ Thus, $\overline{T}_t$ is equivalent to $T_t$.
Moreover, we have
\begin{eqnarray*}
\overline{T}_t(v)&=&({\Id}_\g-L_xt+L^2_xt^2+\cdots+(-1)^iL^i_xt^{i}+\cdots)(T_t(v+\rho^L(x)vt))\\
            &=&T(v)+(\frkT_1v+T\rho^L(x)(v)-[x,Tv]_\g)t+\bar{\tau}_2(v)t^2+\cdots\\
            &=&T(v)+\bar{\frkT}_2(v)t^2+\cdots.
\end{eqnarray*}
Then by repeating the argument,  we find that $T_t$ is equivalent to $T$. \end{proof}

\subsection{Deformations of order $n$ of a relative Rota-Baxter operator}
We introduce a cohomology class associated to any deformation of order $n$ of a relative Rota-Baxter operator, and show that a deformation of order $n$ of a relative Rota-Baxter operator is extensible if and only if this cohomology class is trivial. Thus we call this cohomology class the obstruction class of a deformation of order $n$ being extensible.

\begin{defi}Let $T$ be  a relative Rota-Baxter operator on a Leibniz algebra $(\g,[\cdot,\cdot]_\g)$ with respect to a representation $(V;\rho^L,\rho^R)$.
If $T_t=\sum_{i=0}^n\frkT_i t^i$ with $\frkT_0=T$,
$\frkT_i\in\Hom_{\K}(V,\g)$, $i=2,\cdots, n$, defines a
$\K[t]/(t^{n+1})$-module map from $V[t]/(t^{n+1})$ to the Lie
algebra $\g[t]/(t^{n+1})$  satisfying
\begin{eqnarray}
[T_t(u),T_t(v)]_\g=T_t\Big(\rho^L(T_t(u))v+\rho^R(T_t(v))u\Big),\;\;\forall
u,v\in V, \mlabel{O-operator of order n}
\end{eqnarray} we say that $T_t$
is an {\bf  order $n$ deformation} of the relative Rota-Baxter operator $T$.
\end{defi}

\begin{rmk} Obviously, the left hand side of Eq.~(\ref{O-operator of order
n}) holds in the Lie algebra $\g[t]/(t^{n+1})$ and the  right  hand
side makes sense since $T_t$ is a $\K[t]/(t^{n+1})$-module
map.
\end{rmk}

\begin{defi}
Let $T_t=\sum_{i=0}^n\frkT_i t^i$ be an order $n$ deformation  of
a relative Rota-Baxter operator $T$ on a Leibniz algebra $(\g,[\cdot,\cdot]_\g)$ with respect to a representation $(V;\rho^L,\rho^R)$.
 If there
exists a $1$-cochain $\frkT_{n+1}\in \Hom_{\K}(V,\g)$ such that
$\widetilde{T}_t=T_t+\frkT_{n+1}t^{n+1}$ is an order $n+1$
deformation of the relative Rota-Baxter operator $T$, then we say that $T_{t}$
is {\bf extendable}.
\end{defi}

\begin{pro}
Let $T_t=\sum_{i=0}^n\frkT_i t^i$ be an order $n$
deformation  of a  relative Rota-Baxter operator $T$ on a Leibniz algebra $(\g,[\cdot,\cdot]_\g)$ with respect to a representation $(V;\rho^L,\rho^R)$. Define $\Ob_T \in
C^2(V,\g)$ by
\begin{eqnarray}
\Ob_T(u,v)=\sum\limits_{i+j=n+1\atop
i,j\geq1}\Big([\frkT_iu,\frkT_jv]_\g-\frkT_i\big(\rho^L(\frkT_ju)v+\rho^R(\frkT_jv)u\big)\Big),\;\;
\forall u,v\in V. \mlabel{ob}
\end{eqnarray}
Then the 2-cochain $\Ob_T$ is a $2$-cocycle, that is, $\partial_T
\Ob_T=0$.
\end{pro}

\begin{proof}
By the bracket in Eq.~\eqref{gla-rota-baxter-leibniz},
we have $\Ob_T=\frac{1}{2}\sum\limits_{i+j=n+1\atop
i,j\geq1}\{\frkT_i,\frkT_j\}. $ Since $T_t$ is an order
$n$ deformation  of the  relative Rota-Baxter operator $T$, for all $0\leq i\leq
n$, we have
\begin{eqnarray}
\sum\limits_{k+l=i\atop
k,l\geq0}\Big([\frkT_ku,\frkT_lv]_\g-\frkT_k\big(\rho^L(\frkT_lu)v+\rho^R(\frkT_lv)u\big)\Big)=0, \quad\forall u,v\in V,
\mlabel{deformation-1}
\end{eqnarray}
which is equivalent to
\begin{eqnarray}
\frac{1}{2}\sum\limits_{k+l=i\atop k,l\geq1}\{\frkT_k,\frkT_l\}=-\{T,\frkT_{i}\},\quad 0\leq i\leq
n.
\mlabel{deformation-2}
\end{eqnarray}
By Theorem \ref{partial-to-derivation} and \eqref{eq:huaO-dT}, we have
\begin{eqnarray*}
\partial_T \Ob_T&=&(-1)^1\{T,\Ob_T\}\\
             &=&-\frac{1}{2}\sum\limits_{i+j=n+1\atop i,j\geq1}\{T,\{\frkT_i,\frkT_j\}\}\\
             &=&-\frac{1}{2}\sum\limits_{i+j=n+1\atop i,j\geq1}\Big(\{\{T,\frkT_i\},\frkT_j\}-\{\frkT_i,\{T,\frkT_j\}\}\Big)\\
             &\stackrel{\eqref{deformation-2}}{=}&\frac{1}{4}\sum\limits_{i'+i''+j=n+1\atop i',i'',j\geq1}\{\{\frkT_{i'},\frkT_{i''}\},\frkT_j\}-\frac{1}{4}
       \sum\limits_{i+j'+j''=n+1\atop i,j',j''\geq1}\{\frkT_i,\{\frkT_{j'},\frkT_{j''}\}\}\\
             &=&\frac{1}{2}\sum\limits_{i'+i''+j=n+1\atop i',i'',j\geq1}\{\{\frkT_{i'},\frkT_{i''}\},\frkT_j\}
             =0.
\end{eqnarray*}
Thus, we obtain that the 2-cochain $\Ob_T$ is a $2$-cocycle. The proof is finished.
\end{proof}

 \begin{defi}
  Let $T_t=\sum_{i=0}^n\frkT_i t^i$ be an order $n$ deformation  of a  relative Rota-Baxter operator $T$ on a Leibniz algebra $(\g,[\cdot,\cdot]_\g)$ with respect to a representation $(V;\rho^L,\rho^R)$.  The cohomology
class $[\Ob_T]\in \huaH^2(V,\g)$ is called the {\bf obstruction
class} of  $T_t$ being extendable.
  \vspace{-.1cm}
  \end{defi}

\begin{thm}\label{thm:extendable}
Let $T_t=\sum_{i=0}^n\frkT_i t^i$ be an order $n$ deformation of a  relative Rota-Baxter operator $T$ on a Leibniz algebra $(\g,[\cdot,\cdot]_\g)$ with respect to a representation $(V;\rho^L,\rho^R)$. Then
$T_t$ is extendable if and only if the obstruction class $[\Ob_T]$ is trivial.
\end{thm}
\begin{proof}
Suppose that an order $n$ deformation $T_t$   of the relative Rota-Baxter operator $T$ extends to an order $n+1$ deformation. Then Eq. \eqref{deformation-2} holds for $i=n+1$. Thus, we have
$
\Ob_T=-\partial_T\frkT_{n+1},
$
which implies that the obstruction class $[\Ob_T]$ is trivial.

Conversely, if the obstruction class $[\Ob_T]$ is trivial, suppose that
$
\Ob_T=-\partial_T\frkT_{n+1}
$
for some 1-cochain $\frkT_{n+1} \in\Hom_{\K}(V,\g)$. Set
$
\widetilde{T}_t:=T_t+\frkT_{n+1}t^{n+1}.
$
Then $\widetilde{T}_t$ satisfies Eq.~\eqref{deformation-1} for $0\leq i\leq n+1$. So $\widetilde{T}_t$ is an order $n+1$ deformation, which means that  $T_t$ is extendable.
  \vspace{-.1cm}
\end{proof}

\begin{cor} Let $T$ be a  relative Rota-Baxter operator $T$ on a Leibniz algebra $(\g,[\cdot,\cdot]_\g)$ with respect to a representation $(V;\rho^L,\rho^R)$. If
$\huaH^2(V,\g)=0$, then every $1$-cocycle in $\huaZ^1(V,\g)$ is
the infinitesimal of some formal deformation of the
relative Rota-Baxter operator $T$.

\end{cor}

\noindent
{\bf Acknowledgements. } This research was partially supported by NSFC (11922110).


\begin{thebibliography}{abc}

%\bibitem{Aguiar}
%M. Aguiar, Pre-Poisson algebras. {\em Lett. Math. Phys.} 54 (2000), no. 4, 263-277.

%\bibitem{AO} Sh. A. Ayupov and B. A. Omirov, On Leibniz algebras. \emph{Algebra and operator theory (Tashkent, 1997)}, 1-12, Kluwer Acad. Publ. Dordrecht, 1998.

\bibitem{AMM}
E. Abdaoui, S. Mabrouk and  A. Makhlouf, Rota-Baxter operators on pre-Lie superalgebras. \emph{Bull. Malays. Math. Sci. Soc.}  {\bf42} (2019), 1567-1606.

\bibitem{ABM}
A. Arfa, N. Ben Fraj and A. Makhlouf, Cohomology and deformations of $n$-Lie algebra morphisms. \emph{J. Geom. Phys.} {\bf 132} (2018), 64-74.

\bibitem{Ba} G. Baxter, An analytic problem whose solution follows from a simple algebraic identity. \emph{Pacific J. Math.} {\bf 10} (1960), 731-742.

%\bibitem{Bai-3}
%C. Bai, Left-symmetric bialgebras and an analogue of the classical Yang-Baxter equation. \emph{Commun. Contemp. Math.} 10 (2008), no. 2, 221-260.

%\bibitem{Bai-Sheng-Zhu}
%C. Bai, Y. Sheng and C. Zhu, Lie 2-bialgebras. \emph{Comm. Math. Phys.} 320 (2013), no. 1, 149-172.

\bibitem{Bai-Bellier-Guo-Ni}
C. Bai, O. Bellier, L. Guo and X. Ni, Splitting of operations, Manin products, and Rota-Baxter operators. \emph{Int. Math. Res. Not.} {\bf 3} (2013), 485-524.

%\bibitem{Bai-Guo-Sheng}
%C. Bai, L. Guo and Y. Sheng, Bialgebras, the classical Yang-Baxter equation and Manin triples for $3$-Lie algebras. \emph{arXiv:1604.05996}.

%\bibitem{Balavoine}
%D. Balavoine, Homology and cohomology with coefficients, of an algebra over a quadratic operad. {\em J. Pure Appl. Algebra} 132 (1998), no. 3, 221-258.

\bibitem{Balavoine-1}
D. Balavoine, Deformation of algebras over a quadratic operad. {\em Contemp. Math.} {\bf 202}  (1997), 207-234.

%\bibitem{BaR} R. J. Baxter, One-dimensional anisotropic Heisenberg chain. \emph{Ann. Physics} {\bf 70} (1972), 323-337.

%\bibitem{BW}M. Bordemann and F. Wagemann,   Global integration of Leibniz algebras. \emph{J. Lie Theory} {\bf 27} (2017), 555-567.

\bibitem{CK}
A. Connes and D. Kreimer, { Renormalization in quantum field theory and the Riemann-Hilbert problem. I. The Hopf algebra structure of graphs and the main theorem.} {\em Comm. Math. Phys.} {\bf 210} (2000), 249-273.

\bibitem{CP}
V. Chari and A. Pressley, A Guide to Quantum Groups. Cambridge University Press, 1994.

%\bibitem{Int1}
%S. Covez,   The local integration of Leibniz algebras. \emph{Ann. Inst. Fourier (Grenoble)} {\bf  63} (2013), 1-35.

%\bibitem{Dz}
% A. Dzhumadil$\rm{'}$daev, Cohomologies and deformations of right-symmetric algebras. \emph{J. Math. Sci.} {\bf 93} (1999), 836-876.

%\bibitem{DW}
%B. Dherin and F. Wagemann,   Deformation quantization of Leibniz algebras. \emph{Adv. Math.} {\bf270} (2015), 21-48.

\bibitem{Fregier-Zambon-2}
Y. Fr\'egier, and M. Zambon, Simultaneous deformations of algebras and morphisms via derived brackets. \emph{J. Pure Appl. Algebra} {\bf 219 } (2015), 5344-5362.

\bibitem{Fregier-Zambon-1}
Y. Fr\'egier, and M. Zambon, Simultaneous deformations and Poisson geometry. \emph{ Compos. Math.} {\bf 151} (2015), 1763-1790.

\bibitem{GKO}
S. G\'omez-Vidal, A. Khudoyberdiyev and B. A. Omirov, Some remarks on semisimple Leibniz algebras. \emph{J. Algebra}  {\bf410} (2014), 526-540.

\bibitem{Ge0}
M. Gerstenhaber, The cohomology structure of an associative ring. \emph{Ann. Math.} {\bf 78} (1963) 267-288.

\bibitem{Ge}
M. Gerstenhaber, On the deformation of rings and algebras. \emph{Ann. Math. (2) } {\bf 79} (1964), 59-103.

%\bibitem{Ge2}
%M. Gerstenhaber,   On the deformation of rings and algebras. II. \emph{Ann. Math.} {\bf 84} (1966), 1-19.

%\bibitem{Ge3}
%M. Gerstenhaber,   On the deformation of rings and algebras. III. \emph{Ann. Math.} {\bf 88} (1968), 1-34.

%\bibitem{Ge4}
%M. Gerstenhaber,   On the deformation of rings and algebras. IV. \emph{Ann. Math.} {\bf 99} (1974), 257-276.

\bibitem{Gub-AMS}
L. Guo,  What is a Rota-Baxter algebra? {\em Notices of the AMS} {\bf 56}  (2009), 1436-1437.

\bibitem{Gub}
L. Guo,  An introduction to Rota-Baxter algebra. Surveys of Modern Mathematics, 4. International Press, Somerville, MA; Higher Education Press, Beijing, 2012. xii+226 pp.

%\bibitem{KSo}
%M. Kontsevich and Y. Soibelman, Deformation theory. I [Draft], http://www.math.ksu.edu/~soibel/Book-vol1.ps, 2010.

%\bibitem{KS}
%A. Kotov and T. Strobl, The Embedding Tensor, Leibniz-Loday Algebras, and Their Higher Gauge Theories. \emph{Comm. Math. Phys.} (2020),  https://doi.org/10.1007/s00220-019-03569-3.


%\bibitem{Kosmann-Schwarzbach}
% Y. Kosmann-Schwarzbach, From Poisson algebras to Gerstenhaber algebras. {\em Ann. Inst. Fourier (Grenoble) } {\bf 46} (1996), 1243-1274.

%\bibitem{liu} J. Liu, Twisting on pre-Lie algebras and quasi-pre-Lie algebras.  \emph{  arXiv:2003.11926.}

\bibitem{Loday}
J.-L. Loday, Une version non commutative des alg\`{e}bres de Lie: les alg\`{e}bres de Leibniz. \emph{Enseign. Math.} (2), {\bf  39 } (1993), 269-293.

\bibitem{Loday and Pirashvili}
J.-L. Loday and T. Pirashvili, Universal enveloping algebras of Leibniz algebras and (co)homology. \emph{Math. Ann.}  {\bf 296 } (1993), 139-158.

\bibitem{Makhlouf and Silvestrov}
A. Makhlouf and S. Silvestrov, Notes on $1$-parameter formal deformations of Hom-associative and Hom-Lie algebras. \emph{Forum Math.} {\bf 22} (2010), 715-739.
%\bibitem{LV} J.-L. Loday and B. Vallette, Algebraic Operads. Springer, 2012.

%\bibitem{Loday95}
%J. L. Loday, Cup-product for Leibniz cohomology and dual Leibniz algebras.
%{\em Math. Scand. }77 (1995), no. 2, 189-196.

\bibitem{NR} A. Nijenhuis  and R. Richardson,  Cohomology and deformations in graded Lie algebras. {\em Bull.
Amer. Math. Soc.} {\bf 72} (1966) 1-29.

\bibitem{NR2} A. Nijenhuis and R. Richardson,  Commutative  algebra cohomology and deformations of Lie and associative algebras. {\em J. Algebra} {\bf 9} (1968) 42-105.

%\bibitem{Pei-Bai-Guo-Ni}
%J. Pei, C. Bai, L. Guo and X. Ni, Replicating of binary operads, Koszul duality, Manin products and average operators.   arXiv:1212.0177.

%\bibitem{PG}
%J. Pei and L. Guo,   Averaging algebras, Schroder numbers, rooted trees and operads. \emph{J. Algebraic Combin.} 42 (2015), no. 1, 73-109.

%\bibitem{Sheng-Bai}
%Y. Sheng and C. Bai, A new approach to hom-Lie bialgebras. {\em J. Algebra} 399 (2014), 232-250.

%\bibitem{Sheng-Liu}Y. Sheng and Z. Liu,  Remarks on Leibniz algebras,  arXiv:1408.2225.

%\bibitem{Rota}
%G. C. Rota, Baxter operators, an intoduction, in Gian-Carlo Rota on Combinatics: Introductory papers and commentaries (Joseph P.S. Kung, Ed), Birkh¡§auser, Boston (1995).

\bibitem{Omirov}
B. A. Omirov, Conjugacy of Cartan subalgebras of complex finite-dimensional Leibniz algebras. \emph{J. Algebra} {\bf302} (2006), 887-896.

\bibitem{PBG} J. Pei, C. Bai and L. Guo, Splitting of Operads and Rota-Baxter Operators on Operads. \emph{Appl. Categor. Struct.} {\bf 25} (2017), 505-538.

%\bibitem{Roytenberg}
 %D. Roytenberg, Courant algebroids, derived brackets and even symplectic supermanifolds.  Ph.D. thesis, University of California Berkeley, 1999, arXiv:math/9910078.

%\bibitem{SW}T. Strobl and F. Wagemann, Enhanced Leibniz Algebras: Structure Theorem and Induced Lie 2-Algebra. \emph{Comm. Math. Phys.} (2020), https://doi.org/10.1007/s00220-019-03522-4.

\bibitem{STS} M. A. Semonov-Tian-Shansky, What is a classical R-matrix? \emph{Funct. Anal. Appl.} {\bf 17} (1983) 259-272.

\bibitem{ST}Y. Sheng and R. Tang, Leibniz bialgebras, relative Rota-Baxter operators and the classical Leibniz  Yang-Baxter equation. {arXiv:1902.03033}.

\bibitem{Tang-Bai-Guo-Sheng}
R. Tang, C. Bai, L. Guo and Y. Sheng, Deformations and their controlling cohomologies of $\huaO$-operators.  \emph{Comm. Math. Phys.} {\bf 368} (2019), 665-700.

%\bibitem{Ta} L. Takhtajan, On foundation of the generalized Nambu mechanics. \emph{Comm. Math. Phys.} {\bf 160} (1994), 295-315.

%\bibitem{Uchino08}
%K. Uchino,   Quantum analogy of Poisson geometry, related dendriform algebras and Rota-Baxter operators. \emph{Lett. Math. Phys.} 85 (2008), no. 2-3, 91-109.

%\bibitem{Uchino-2}
%K. Uchino, Derived bracket construction and Manin products, \emph{Lett. Math. Phys.} 93 (2010) 37-53.

%\bibitem{Vallette}
%B. Vallette, Manin products, Koszul duality, Loday algebras and Deligne conjecture.  \emph{J. Reine Angew. Math.} 620 (2008) 105-164.

%\bibitem{Vo} Th. Voronov, Higher derived brackets and homotopy algebras. \emph{J. Pure Appl. Algebra}  {\bf 202} (2005), 133-153.

%\bibitem{Ya} C. N. Yang, Some exact results for the many-body problem in one dimension with repulsive delta-function interaction. \emph{Phys. Rev. Lett.} {\bf 19} (1967), 1312-1315.

\end{thebibliography}
\end{document}